\pdfoutput=1
\documentclass{amsart}
\usepackage{amsthm}
\usepackage{amssymb}
\usepackage{colonequals}
\usepackage{dsfont}
\usepackage{enumerate}
\usepackage{eucal}

\usepackage{tikz-cd}
\tikzset{labl/.style={anchor=south, rotate=270, inner sep=.5mm}}

\usepackage[pagebackref]{hyperref}
\hypersetup{%
  bookmarksnumbered=true,%
  colorlinks=true,%
  linkcolor=blue,%
  citecolor=blue,%
  filecolor=blue,%
  menucolor=blue,%
  urlcolor=blue,%
  bookmarksopen=true,%
  bookmarksdepth=2,%
  pageanchor=true}

\makeatletter
\@namedef{subjclassname@2020}{\textup{2020} Mathematics Subject Classification}
\makeatother

\numberwithin{equation}{section}

\theoremstyle{plain}
\newtheorem*{itheorem}{Theorem}
\newtheorem{theorem}[equation]{Theorem}
\newtheorem{proposition}[equation]{Proposition}
\newtheorem{lemma}[equation]{Lemma} 
\newtheorem{corollary}[equation]{Corollary}

\theoremstyle{definition}

\newtheorem{chunk}[equation]{}

\theoremstyle{remark}
\newtheorem{remark}[equation]{Remark} 

\newtheorem*{ack}{Acknowledgements}

\newcommand{\bSpec}{\operatorname{Spc}}

\newcommand{\cat}{\mathcal}
\newcommand{\comp}[1]{{#1}^{\operatorname{c}}}
\newcommand{\dbcat}[1]{{\mathbf D}^{\mathrm{b}}(\operatorname{mod}#1)}
\newcommand{\dcat}[1]{{\mathbf D}(#1)}
\newcommand{\duals}[1]{{#1}^{\operatorname{d}}}

\newcommand{\Ext}{\operatorname{Ext}}

\newcommand{\gam}{\varGamma} 
\newcommand{\Hom}{\operatorname{Hom}}
\newcommand{\fHom}{\operatorname{\mathcal{H}\!\!\;\mathit{om}}}

\newcommand{\lam}{\varLambda}
\newcommand{\llam}{{\mathbf L}\varLambda}
\newcommand{\lotimes}{\otimes^{\mathbf L}}

\newcommand{\Mod}{\operatorname{Mod}}

\newcommand{\one}{\mathds 1}
\newcommand{\op}{\mathrm{op}}
\newcommand{\Proj}{\operatorname{Proj}}
\newcommand{\rank}{\operatorname{rank}}
\newcommand{\rgam}{{\mathbf R}\Gamma}
\newcommand{\RHom}{\operatorname{{\mathbf R}Hom}}

\newcommand{\Si}{\Sigma}

\newcommand{\Spec}{\operatorname{Spec}}
\newcommand{\StMod}{\operatorname{StMod}}
\newcommand{\stmod}{\operatorname{stmod}}
\newcommand{\supp}{\operatorname{supp}}
\newcommand{\swd}{D^{\scriptscriptstyle{\mathrm {SW}}}}

\newcommand{\Thick}{\operatorname{Thick}}
\newcommand{\wh}{\widehat}

\newcommand{\iso}{\xrightarrow{\raisebox{-.4ex}[0ex][0ex]{$\scriptstyle{\sim}$}}}
\newcommand{\longiso}{\xrightarrow{\ \raisebox{-.4ex}[0ex][0ex]{$\scriptstyle{\sim}$}\ }}
\newcommand{\lra}{\longrightarrow}
\newcommand{\xra}{\xrightarrow}

\newcommand{\bfq}{\mathbf q} 
\newcommand{\bfD}{\mathbf D} 
\newcommand{\bfK}{\mathbf K}

\newcommand{\bbZ}{\mathbb Z} 
\newcommand{\bbX}{\mathbb X}
\newcommand{\bfp}{\mathbf{p}}

\newcommand{\fm}{\mathfrak{m}} 
\newcommand{\fp}{\mathfrak{p}}
\newcommand{\fq}{\mathfrak{q}} 
\newcommand{\eps}{\varepsilon}

\title[Local dualisable objects]{Local dualisable objects \\ in local algebra}

\author[Benson, Iyengar, Krause, and Pevtsova]{Dave Benson, Srikanth
  B. Iyengar, Henning Krause \\ and Julia Pevtsova}

\address{Dave Benson \\ 
Institute of Mathematics\\ 
University of Aberdeen\\ 
King's College\\ 
Aberdeen AB24 3UE\\ 
Scotland U.K.}

\address{Srikanth B. Iyengar\\ 
Department of Mathematics\\
University of Utah\\ 
Salt Lake City, UT 84112\\ 
U.S.A.}

\address{Henning Krause\\ 
Fakult\"at f\"ur Mathematik\\ 
Universit\"at Bielefeld\\ 
33501 Bielefeld\\ 
Germany.}

\address{Julia Pevtsova\\ 
Department of Mathematics\\ 
University of Washington\\ 
Seattle, WA 98195\\ 
U.S.A.}

\begin{document}

\begin{abstract}
We discuss dualisable objects in minimal subcategories of compactly generated tensor triangulated categories, paying special attention to the derived category of a commutative noetherian ring. A cohomological criterion for detecting these local dualisable objects is established. Generalisations to other related contexts are discussed. 
\end{abstract}

\keywords{Balmer spectrum, compact object, derived category, dualisable object, reflexive object, tensor triangulated category}

\subjclass[2020]{13D09 (primary); 18G80, 14F08 (secondary)}

\date{\today}

\maketitle

\setcounter{tocdepth}{1}

\section{Introduction}
Let $\cat T$ be a rigidly compactly generated tensor triangulated category; in this work we consider only the symmetric tensor categories. A central problem  is to classify the localising tensor ideals in $\cat T$. Consider the lattice, with respect to inclusion, of such subcategories. In many contexts, its structure is determined by the minimal elements in the lattice.  Often these minimal elements are parameterised by some topological space; for instance, the Balmer spectrum, or the spectrum of some commutative ring acting on $\cat T$, in the sense of \cite{Benson/Iyengar/Krause:2011a}.  

We are interested in the structure of a minimal subcategory, say
$\cat S$. Minimality implies that there are no proper localising
tensor ideals in $\cat S$. In particular, there are no proper thick
tensor ideals in the subcategory of compact objects in $\cat S$.
Typically however, there are dualisable objects which are not
compact. Thus there is a collection of thick tensor ideals in the subcategory of
dualisable objects in $\cat S$, and one can get a handle on them by
computing its spectrum.

This is what is done in this work for $\dcat A$, the derived category of a commutative noetherian ring $A$. In that case the minimal localising subcategories are the subcategories, $\Gamma_\fp\dcat A$, of the derived category consisting of the $\fp$-local and $\fp$-torsion $A$-complexes, where $\fp$ is a prime ideal in $A$. Our main result, which reappears as Theorem~\ref{th:ca-ring}, characterises dualisable objects in these categories.

\begin{itheorem}
\label{itheorem}
For each $X$ in $\Gamma_\fp\dcat A$ the following conditions are equivalent.
\begin{enumerate}[\quad\rm(1)]
\item
$X$ is dualisable in $\Gamma_\fp\dcat A$;
\item
$\rank_{k(\fp)}H(k(\fp)\lotimes_{A}X)$ is finite;
\item
$X$ is in $\Thick(\rgam_{\fp}(A_\fp))$.
\end{enumerate}
\end{itheorem}

In the statement $k(\fp)$ is the residue field of the local ring $A_\fp$ and $\rgam_{\fp}(-)$ is the local cohomology functor with respect to $\fp$; see Section~\ref{se:ca-rings}. In contrast, an object $X$ in $\Gamma_\fp\dcat A$ is compact precisely when $\mathrm{length}_{A_\fp} H(X)$ is finite. 

It follows from the characterisation above that the dualisable
complexes in $\Gamma_\fp\dcat A$ form a thick subcategory; they are
always closed under tensor products. Given this, and the equivalence
of complete modules and torsion modules, established by Dwyer and
Greenlees~\cite{Dwyer/Greenlees:2002a}, we deduce that the spectrum of
the dualisable objects identifies with the Zariski spectrum of the
completion of the local ring $A_\fp$ at its maximal ideal; see
Corollary~\ref{co:balmer}. This suggests viewing the passage from
compact objects to dualisable objects in any compactly generated
tensor triangulated category as a completion process. Similar
considerations imply that when the ring $A_\fp$ is regular, the
category of dualisable objects in $\Gamma_\fp\dcat A$ has a strong
generator; see Corollary~\ref{co:regularity} and the discussion
surrounding it.

Here is an outline of the contents of this manuscript: Section~\ref{se:dualisability} collects some well-known, though not well-recorded, results and remarks on notions of smallness and dualisablity in general compactly generated tensor triangulated categories. Sections~\ref{se:ca-rings} and \ref{se:local-dualisable} are about the derived category of a commutative noetherian ring,  culminating in the characterisation of local dualisable objects and some corollaries. Section~\ref{se:others} contains a discussion of dualisable objects in other contexts, including the stable homotopy category. In fact, the formulation of the theorem above and our work reported here and in \cite{Benson/Iyengar/Krause/Pevtsova:2023b} is inspired by work of Hovey and Strickland~\cite{Hovey/Strickland:1999a} on dualisable objects in the $K(n)$-local stable homotopy category.

\begin{ack}
Part of this work was done during the Trimester Program ``Spectral
Methods in Algebra, Geometry, and Topology" at the Hausdorff Institute
in Bonn. It is a pleasure to thank HIM for hospitality and for funding
by the Deutsche Forschungsgemeinschaft under Excellence
Strategy EXC-2047/1-390685813. Our work also benefitted from participation in the Abel Symposium on ``Triangulated categories in representation theory and beyond", held in {\AA}lesund, in June 2022. We thank the Niels Henrik Abel Memorial Fund for invitation to this event. During the preparation of this work, SBI was partly supported by NSF grants DMS-1700985 and DMS-2001368. JP was partly supported by NSF grants DMS-1901854 and DMS-2200832, and a Brian and Tiffinie Pang faculty fellowship.
\end{ack}

\section{Dualisability}
\label{se:dualisability}
Though the focus of our work is on the derived category of a commutative ring, we begin by recalling various notions of dualisability in general tensor triangulated categories. Our basic references for this material are \cite[Appendix~A.2]{Hovey/Palmieri/Strickland:1997a} and \cite[Chapter~III]{Lewis/May/Steinberger:1986a}.  While much of the discussion is valid for symmetric monoidal categories, our examples are equipped with a compatible structure of a triangulated category so we work in that context.

Let us fix a compactly generated tensor triangulated category $({\cat T},\otimes,\one)$, with symmetric tensor product $\otimes$ and unit $\one$; the latter need not be compact. As usual, $\cat T^{\mathrm c}$ denotes the full subcategory of compact objects in $\cat T$.

Brown representability yields \emph{function objects} $\fHom(X,Y)$ satisfying an adjunction isomorphism
\[
\Hom_{\cat T}(X\otimes Y,Z)\cong \Hom_{\cat T}(X,\fHom(Y,Z)) \quad \text{for all $X,Y,Z$ in ${\cat T}$.}
\]
The construction implies that the functor $\fHom(Y,-)$ on ${\cat T}$ is exact; we will assume that the functor $\fHom(-,Z)$ is also exact. The adjunction isomorphism above yields natural isomorphisms
\[
\fHom(X\otimes Y,Z) \cong \fHom(X,\fHom(Y,Z))\,.
\]
The counit of the adjunction above plays a role in the sequel:
\[
 \eps \colon \fHom(X,Y)\otimes X \lra Y\,.
\]
We will need the symmetric braiding in ${\cat T}$ that we denote:
\[
\gamma \colon X\otimes Y \longiso Y\otimes X\,.
\]
One has also a natural map
\begin{equation}
\label{eq:nu}
\nu \colon \fHom(X,Y)\otimes Z\lra \fHom(X,Y\otimes Z)\,,
\end{equation}
obtained as the adjoint to the composition of maps
\[
\fHom(X,Y)\otimes Z\otimes X \xra{\ 1\otimes \gamma\ } \fHom(X,Y)\otimes X\otimes Z \xra{\ \eps \otimes 1\ } Y \otimes Z\,.
\]
The \emph{Spanier--Whitehead dual} of an object $X$ is
\[
\swd X \colonequals \fHom(X,\one)\,.
\]
The assignment $X\mapsto \swd X$ is a contravariant functor ${\cat T}\to{\cat T}$.    An object $X$ in ${\cat T}$ is said to be \emph{dualisable} if for all
$Y$ in ${\cat T}$ the natural map
\[
\swd X\otimes Y\lra \fHom(X,Y)\,,
\]
obtained from \eqref{eq:nu} by setting $Z=\one$, is an isomorphism. We denote by $\cat T^{\mathrm d}$ the full subcategory of dualisable
objects in $\cat T$.

The adjoint of the composite $X\otimes \swd X \xra{\gamma} \swd X\otimes X \xra{\eps}\one$ of the braiding with the counit $\eps$  gives the natural double duality map
\[
\rho \colon X \lra \swd\swd(X)\,.
\]
We say $X$ is \emph{reflexive} if this map is an isomorphism. Dualisable objects are reflexive---this is part of the result below---but not conversely; see~\ref{ch:KProj}.

An object $X$ in $\cat T$ is said to be \emph{functionally compact}
if for all set-indexed collections of objects $\{Y_i\}$ the following
natural map is an isomorphism:
\[
\bigoplus_i \fHom(X,Y_i) \lra \fHom(X, \bigoplus_i Y_i).
\]
Observe that replacing the function object with $\Hom$ defines compactness. 

The result below collects some useful  observations concerning these notions; we give the proofs because some of the arguments are rather delicate, and  not easy to find in the literature. We also invite the reader to verify these statements directly for the derived category of a commutative noetherian ring.

\begin{proposition}
\label{pr:dualisable}
Let $X$ be an object in ${\cat T}$. The following statements hold.
\begin{enumerate}[\quad\rm(1)]
\item
The object $X$ is dualisable if and only if there is a map $\eta\colon \one \to X\otimes \swd X$ making the following diagram commute
\[
\begin{tikzcd}
\one \arrow{d} \arrow["\eta"]{r} & X\otimes \swd X \arrow["\gamma"]{d} \\
\fHom(X,X) \arrow[leftarrow]{r} & \swd X\otimes X
\end{tikzcd}
\]
The vertical map on the left is the adjoint to the isomorphism $\one \otimes X \iso X$.
\item
If $X$ is dualisable so is $\swd X$ and $\rho\colon X\to \swd\swd X$ is an isomorphism.
\item
If either  $X$ or $Z$ is dualisable, then the map \eqref{eq:nu} is an isomorphism.
\item
If $X$ is dualisable and $C\in{\cat T}$ is compact, then $C\otimes X$ is compact.
\item
If $X$ is dualisable it is functionally compact; the converse holds if ${\cat T}$ is generated by a set of dualisable objects.
\item
If $X$ is functionally compact and $\one$ is compact, then $X$ is compact.
\end{enumerate}
\end{proposition}

\begin{proof}
(1) When $X$ is dualisable, the map $\swd X\otimes X\to \fHom(X,X)$ is an isomorphism, and we can use its inverse to get a map $\eta\colon \one \to X\otimes \swd X$, and this fits into the commutative diagram as desired. Conversely, given such an $\eta$ a diagram chase shows that the composite
\begin{multline*}  
\fHom(X,Y)  \iso \fHom(X,Y)\otimes\one 
	            \xra{1\otimes \eta} \fHom(X,Y)\otimes X\otimes \swd X \\
	           \xra{\eps\otimes 1}Y\otimes \swd X 	 \xra{\gamma} \swd X\otimes Y
\end{multline*}  
is the inverse of the map $\swd X\otimes Y\to \fHom(X,Y)$.

(2) Given $\eta\colon \one \to X \otimes \swd X$ as in (1), the composite 
\[ 
\one \xrightarrow{\eta} X \otimes   \swd{X} \xrightarrow{\rho\otimes 1}
  \swd\swd{X} \otimes  \swd{X} \xrightarrow{\gamma}  \swd{X} \otimes  \swd\swd{X}
\]
plays the role of $\eta$ for $ \swd{X}$, so again using (1), $ \swd{X}$ is dualisable. A diagram chase shows that an inverse for $\rho\colon X \to  \swd\swd{X}$ is given by the composite
\begin{multline*}  
\swd\swd{X} \iso \one\otimes \swd\swd{X} \xrightarrow{\eta\otimes 1}
X \otimes \swd{X} \otimes \swd\swd{X} \\
	\xrightarrow{1 \otimes \gamma} X \otimes \swd\swd{X} \otimes \swd{X} \xrightarrow{1\otimes \eps}
X \otimes \one \iso X\,.
\end{multline*}

(3) If $X$ is dualisable, then an inverse for $\nu$ is given by the composite
\begin{multline*} 
\fHom(X,Y\otimes Z)\iso
\fHom(X,Y\otimes Z) \otimes \one \xrightarrow{1\otimes \eta}
\fHom(X,Y\otimes Z) \otimes X \otimes \swd{X} \\ 
\xrightarrow{\eps\otimes 1} 
Y \otimes Z \otimes \swd{X} \xrightarrow{\gamma}
\swd{X} \otimes Y \otimes Z \xrightarrow{\nu\otimes 1} \fHom(X,Y) \otimes Z\,. 
\end{multline*}

If $Z$ is dualisable then using (2) we have a commutative diagram
\[ 
\begin{tikzcd}
\fHom(X,Y)\otimes Z \arrow[d, "{1\otimes\rho}" swap, "\sim" labl] \arrow["\nu"]{r} 
	& \fHom(X,Y\otimes Z) \arrow[d,  "{(1\otimes\rho)_*}" swap, "\sim" labl] \\
\fHom(X,Y) \otimes \swd\swd{Z}  \arrow[d, "{\gamma}" swap, "\sim" labl] \arrow[r, "\nu"]
	 & \fHom(X,Y\otimes \swd\swd{Z}) \arrow[d, "{\gamma_*}" swap, "\sim" labl] \\
\swd\swd{Z} \otimes\fHom(X,Y) \ar[d, "\nu" swap, "\sim" labl] 
	&\fHom(X,\swd\swd{Z} \otimes Y)\ar[d,"{\nu_*}" swap, "\sim" labl]\\
\fHom(\swd{Z},\fHom(X,Y)) \ar[d,"\sim" labl]
	&\fHom(X,\fHom(\swd{Z},Y)) \ar[d,"\sim" labl]  \\
\fHom(\swd{Z} \otimes X,Y)\ar[r, "{\gamma^*}" swap, "\sim"]
	&\fHom(X \otimes \swd{Z},Y)
\end{tikzcd}
 \]
The vertical maps are all isomorphisms, as is the bottom horizontal map, and therefore so is the top horizontal map.

(4) This follows from the isomorphisms of functors
  \[
 \Hom_{\cat T}(C\otimes X,-)\cong\Hom_{\cat T}(C,\fHom(X,-))
    \cong\Hom_{\cat T}(C,\swd{X}\otimes-)\,.
\]

(5) For any set of objects $\{Y_i\}$ there is a commutative diagram
\[
\begin{tikzcd}
\bigoplus_i \swd{X}\otimes Y_i  \arrow{d} \arrow["\sim"]{r}  & \swd{X} \otimes \bigoplus_i Y_i  \arrow{d}  \\
\bigoplus_i \fHom(X,{Y_i}) \arrow{r}  & \fHom(X,{\bigoplus_i Y_i})
\end{tikzcd}
\]
When $X$ is dualisable, the two vertical maps are isomorphisms and hence so is the  lower horizontal map, and hence $X$ is functionally compact. 

Conversely, suppose $X$ is functionally compact. Then the lower horizontal map in the diagram above is an isomorphism, and so it follows  that the collection of objects $Y$ for which the map $\swd{X}\otimes Y \to \fHom(X,Y)$ is an isomorphism form a localising subcategory of $\cat T$. By (3), it contains the dualisable objects, so when $\cat T$ is generated by such objects, we deduce that $X$ is dualisable.

(6) Apply $\Hom_{\cat T}(\one,-)$ to the isomorphism defining functional compactness.
\end{proof}

We collect some sundry consequences of the preceding result, for later use. 

\begin{remark}
\label{re:dualisable}
Let $(\cat T,\otimes,\one)$ be a compactly generated tensor triangulated category. 
Proposition~\ref{pr:dualisable} implies that when $\one$ is compact
any dualisable object is compact. The inclusion $\duals{\cat T} \subseteq \comp{\cat T}$ may be strict; see~\ref{ch:KProj}. 

The subcategory $\duals{\cat T}$ is thick, and closed under tensor products, function objects, and hence  also under Spanier--Whitehead duality. On
the other hand, the compact objects in $\cat T$ form a thick subcategory, but may not be closed under tensor products or Spanier--Whitehead duality; see~\ref{ch:KProj}. Thus when compact objects and dualisable objects coincide, $\comp{\cat T}$ is a tensor triangulated subcategory of $\cat T$, with unit $\one$ and the same function object.

The condition that  $\comp{\cat T}=\duals{\cat T}$ is equivalent to $\cat T$ having  a set of generators that are both compact and dualisable.   Hovey, Palmieri, and Strickland~\cite{Hovey/Palmieri/Strickland:1997a} call such a category a \emph{unital algebraic stable homotopy category}; Balmer and Favi~\cite{Balmer/Favi:2011a} use the term \emph{rigidly compactly generated category}.
\end{remark}

\section{Commutative noetherian rings}
\label{se:ca-rings}
Next we describe the compactly generated tensor triangulated categories that are the focus of this work. Throughout $A$ is a commutative noetherian ring. We write  $\dcat A$ for the derived category of $A$ and $\dbcat A$ for the subcategory consisting of $A$-complexes $M$ such that the $A$-module $H(M)\colonequals \bigoplus_i H^i(M)$ is finitely generated.

\begin{chunk}
\label{ch:com-tt}
The derived category of $A$ is a compactly generated triangulated category, with compact objects the perfect complexes, namely, those that are isomorphic in $\dcat A$ to bounded complexes of finitely generated projective $A$-modules; equivalently, the objects in $\Thick(A)$. See, for instance, \cite[\S9.2]{Krause:2022a}. One has that
\[
\Thick(A)\subseteq \dbcat A\,;
\]
equality holds if and only if $A$ is regular, that is to say, for each $\fp\in \Spec A$, the local ring $A_\fp$ is regular. This is just a reinterpretation of the classical characterisation, due to Auslander, Buchsbaum, and Serre \cite[Theorem~2.2.7]{Bruns/Herzog:1998a}, of regular local rings as the local rings of finite global dimension, along with the observation, due to Bass and Murthy that, for objects in $\dcat R$, finite projective dimension can be tested locally; see~\cite[Theorem~4.1]{Avramov/Iyengar/Lipman:2010a}.

The derived tensor product, $-\lotimes_A-$ endows $\dcat A$ with a structure of a tensor triangulated category with unit $A$ and function object $\RHom_A(-,-)$. The unit $A$ generates $\dcat A$, and is  compact and dualisable, so compact objects and dualisable objects coincide.

As to the reflexive objects in $\dcat A$: For an object $X$ in $\dbcat{A}$ the natural map $X\to \swd \swd X$ is an isomorphism if and only if $X$ has finite Gorenstein dimension \cite[Theorem~2.4.7]{Christensen:2000a}. Such an $X$ is not necessarily compact. Indeed, when $A$ is Gorenstein any $X$ in $\dbcat{A}$ has finite Gorenstein dimension, but $\Thick(A)=\dbcat A$ if and only if $A$ is regular.
\end{chunk}

\subsection*{Local cohomology and localisation}
Fix a prime ideal $\fp$ in $A$. An $A$-complex $X$ in $\dcat A$ is \emph{$\fp$-local} if the natural map $X\to X_\fp$ is an isomorphism in $\dcat A$. Since localisation is an exact functor, this conditions is equivalently to the condition that the map  $H(X)\to H(X)_\fp$ of $A$-modules is bijective.

An $A$-complex $X$ is \emph{$\fp$-torsion} if $X_\fq\cong 0$ in $\dcat A$ for each $\fq\not\supseteq \fp$. Once again, it is clear that $X$ is $\fp$-torsion if and only $H(X)$ is $\fp$-torsion; equivalently, each $A$-module $H^i(X)$ is $\fp$-torsion. An $A$-module is $\fp$-torsion precisely when, for each $x\in M$ there exists an integer $s\ge 0$ such that $\fp^s\cdot x=0$; this explains the terminology. 

It is straightforward to check that the class of $\fp$-torsion $A$-complexes is a localising subcategory of $\dcat A$. Its inclusion into $\dcat A$ admits a right adjoint, $\rgam_{\fp}(-)$, the classical local cohomology functor with respect to the (Zariski) closed subset of $\Spec A$ defined by $\fp$; see \cite[\S3.5]{Bruns/Herzog:1998a}, and also \cite[\S9]{Benson/Iyengar/Krause:2008a}.
 
We are interested in the class of $\fp$-local $\fp$-torsion objects, namely, the subcategory
\begin{equation}
\label{eq:p-local-p-torsion}
\Gamma_{\fp} {\dcat A} \colonequals \{X\in\dcat A\mid \rgam_{\fp}(X_\fp)\cong X\}.
\end{equation}
This is a localising tensor ideal  in $\dcat A$, and even minimal, in that the only localising subcategory properly contained in $\gam_{\fp} \dcat A$ is $0$, by \cite[Theorem~2.8]{Neeman:1992a}. Said otherwise, $\dcat A$ is \emph{stratified} by the $A$ action on $\dcat A$, in the sense of \cite{Benson/Iyengar/Krause:2011a}. This has the consequence that localising subcategories of $\dcat A$ are in bijection with the subsets of $\Spec(A)$. One can thus view the categories $\Gamma_{\fp} \dcat A$ as the building blocks of the triangulated category $\dcat A$. And so it is of interest to investigate the objects in it. This is what we do in Section~\ref{se:local-dualisable}.

To wrap up this section, we give an example of a compactly generated tensor triangulated category where the unit is compact, so dualisable objects are compact, but not every compact object is dualisable. It also has the feature that the tensor product of compact objects is not always compact.

\begin{chunk}
\label{ch:KProj}
Let $A$ be a commutative noetherian ring and $\bfK(\Proj A)$ the homotopy category of complexes of projective
$A$-modules. This is a compactly generated triangulated category, with a
triangle equivalence
\[
{\dbcat{A}}^\op \longiso  \comp{\bfK(\Proj A)}
\]
given by the assignment $M\mapsto (\bfp M)^*$, where $\bfp M$ is a projective resolution of $M$ and $(-)^*\colonequals \Hom_A(-,A)$; see \cite{Jorgensen:2005a}.

We endow $\bfK(\Proj A)$ with a structure of  a tensor triangulated category with tensor product the usual tensor product over $A$. The unit for this tensor product is $A$. By Brown representability, the inclusion $\bfK(\Proj A)\to\bfK(\Mod A)$ has a right adjoint $\bfq\colon\bfK(\Mod A)\to\bfK(\Proj  A)$. It is easy to verify that $\bfq$ preserves
function objects. Thus
\[
\fHom(X,Y) \cong \bfq \Hom_A(X,Y)\qquad (X,Y\in \bfK(\Proj A))\,.
\]
Evidently $A$ is compact, so dualisable objects in $\bfK(\Proj A)$ are compact.

We claim that the subcategory of dualisable objects in $\bfK(\Proj A)$ is precisely  $\Thick(A)$, the bounded complexes of finitely generated projective modules. 

Indeed, fix a dualisable object; since it is compact we can assume it is of the form $(\bfp M)^*$, for some $M$ in $\dbcat{ A}$.  Moreover since $A$ is noetherian, we can assume $\bfp M$ consists of finitely generated projective $A$-modules, and that $(\bfp M)^i=0$ for $i\gg 0$. Then the Spanier--Whitehead dual of $(\bfp M)^*$ is
\[
\fHom((\bfp M)^*,A) \cong \bfq \Hom_A((\bfp M)^*,A) \cong \bfq (\bfp M) \cong \bfp M 
\]
where the second isomorphism holds because of the structure of $\bfp M$ and the last one holds because $\bfp M$ is already in $\bfK(\Proj A)$. In particular $\bfp M$ is also dualisable, being the Spanier--Whitehead dual of a dualisable object. But then it is also compact. Observe that $\bfp M$ is in $\mathrm{Loc}(A)$, so compactness implies that it is in $\Thick(A)$. It remains to observe that then so is $(\bfp M)^*$.

Suppose now that $A$ is \emph{singular}; this condition is equivalent to the existence of finitely generated $A$-modules of infinite projective dimension.
Then for any $M$ in $\dbcat{A}$ of infinite projective dimension the complex $(\bfp M)^*$ is compact in $\bfK(\Proj A)$ but it is not dualisable. Moreover the Spanier--Whitehead dual of the compact object $(\bfp M)^*$ is $\bfp M$ and this will not be compact, by the argument above. 

For $M,N$ in $\dbcat A$ the natural map is an isomorphism:
\[
(\bfp M)^*\otimes_A (\bfp N)^*\longiso \Hom_A(\bfp M, (\bfp N)^*)\,.
\]
In particular the cohomology of the object on the left is $\Ext_A(M, \RHom_A(N,A))$. When $A$ is singular and Gorenstein the cohomology of any compact object in $\bfK(\Proj A)$ is bounded. However one can find $M,N$ such that the cohomology of  $(\bfp M)^*\otimes_A (\bfp N)^*$ is not bounded, so the tensor product  will not be compact.
\end{chunk}

\section{Local dualisable objects in $\dcat A$}
\label{se:local-dualisable}

Let $A$ be a commutative noetherian ring and $\dcat A$ the derived category of $A$-modules, with the usual structure of a tensor triangulated category; see~\ref{ch:com-tt}.  As noted there, the dualisable objects and compact objects in $\dcat A$ coincide, and are precisely the perfect complexes in $\dcat A$. In this section we focus on the dualisable objects in $\Gamma_{\fp} \dcat A$, the category of $\fp$-local and $\fp$-torsion objects in $\dcat A$, for $\fp$ a prime ideal in $A$; see \eqref{eq:p-local-p-torsion}.

Fix a prime ideal $\fp$.  It is straightforward to verify that when $X$ and $Y$ are $\fp$-local and $\fp$-torsion, so is $X\lotimes_AY$; that is to say, the triangulated category $\Gamma_{\fp}\dcat A$ inherits a tensor product from $\dcat A$. With this tensor product $\Gamma_{\fp}\dcat A$ is tensor triangulated, with unit $\rgam_{\fp}A_\fp$, and function object
\[
\fHom(X,Y) \colonequals \rgam_{\fp}\RHom_A(X,Y)\,.
\]
The thick subcategory of compact objects in $\Gamma_{\fp}\dcat A$ has a simple structure, in that it is minimal.  The unit  $\rgam_{\fp}(A_\fp)$  is compact only when $\fp$ is a minimal prime ideal in $A$. So, typically, there are more dualisable than compact objects in $\Gamma_{\fp} {\dcat A}$.

Here is a characterisation of the  dualisable objects in this category, in terms of their cohomology. We write $k(\fp)$ for  $A_\fp/\fp A_\fp$, the residue field of the local ring $A_\fp$ and $\Si$ for the suspension, or shift, functor in a triangulated category.

\begin{theorem}
\label{th:ca-ring}
Let $A$ be a commutative noetherian ring and $\fp$ a prime ideal in $A$.  For each $\fp$-local and $\fp$-torsion $A$-complex $X$ the following conditions are equivalent.
\begin{enumerate}[\quad\rm(1)]
\item
$X$ is dualisable in $ \Gamma_{\fp}\dcat A$.
\item
$\rank_{k(\fp)}H(k(\fp)\lotimes_{A}X)$ is finite.
\item
$X$ is in $\Thick(\rgam_{\fp}(A_\fp))$.
\end{enumerate}
If moreover $\rank_{k(\fp)}H(k(\fp)\lotimes_{A}X)=1$, then $X\cong \Si^s\rgam_{\fp}(A_\fp)$ for some integer $s$.
\end{theorem}

As will be clear from the proof, the implications (1)$\Rightarrow$(2) and  (3)$\Rightarrow$(1) are elementary to verify. The  implication (2)$\Rightarrow$(3) is the non-trivial one, and its proof takes most of the work in this section; it makes critical use of derived   completions.  There is a simpler proof when the  ring $A_\fp$ has finite global dimension; see \cite{Benson/Iyengar/Krause/Pevtsova:2023b}.

\subsection*{Derived completions}
Given an ideal $I$ in $A$ and an $A$-module $M$, the \emph{$I$-adic completion} of $M$ is the inverse limit
\[
\lam^IM \colonequals \lim_n (\cdots  \twoheadrightarrow M/I^{n+1}M\twoheadrightarrow M/I^nM \twoheadrightarrow \cdots \twoheadrightarrow M/IM)\,,
\]
where the surjections are the natural ones. The canonical maps $M\to M/I^nM$ induce a map $M\to \lam^IM$; when this is bijective we say $M$ is \emph{classically $I$-complete}. 

Given an $A$-complex $M$ we write $\llam^IM$ for the left derived functor of the completion; see \cite{Greenlees/May:1992a}.  This comes equipped with a morphism $M\to \llam^IM$ in $\dcat A$, and the complex $M$ is said to be \emph{$I$-complete} when this map is a quasi-isomorphism. A complex $M$ is $I$-complete if and only if  $H^i(M)$ is $I$-complete for each $i$.  A caveat: classically complete $A$-modules are complete, but the converse does not hold; see \cite[Example~1.4]{Greenlees/May:1992a} and also \cite[Example~2.4]{Bhatt:2019a}.

When $M$ is an $A$-module, there is natural surjective map $H^0(\llam^IM)\to \lam^IM$. This is an isomorphism when $M$ is a finitely generated, and then $H^i(\llam^IM)=0$ for $i\ge 1$, that is to say, there is an isomorphism $\llam^IM\cong \lam^I M$ in $\dcat A$ for any finitely generated $A$-module $M$. 
In particular, $\llam^IA\cong \lam^IA$; this observation is used implicitly in the sequel.

The derived local cohomology functor $\rgam_I$ and the derived $I$-adic completion functor $\llam^{I}$ form an adjoint pair:
\begin{equation}
\label{eq:GM}
\begin{tikzcd}
\dcat A \arrow[rightarrow,yshift=-.75ex,swap,rr,"\llam^I"]
 	&& \arrow[rightarrow,yshift=.75ex,ll,swap,"\rgam_I"] \dcat A\,.
\end{tikzcd}
\end{equation}
This is the Greenlees-May duality.  It restrict to an equivalence between the $I$-torsion  and $I$-complete complexes, and so one has natural isomorphisms
\begin{equation}
\label{eq:gam-lam}
\rgam_I M \cong \rgam_I \llam^I M \qquad\text{and}\qquad  \llam^I \rgam_I M\cong \llam^I M \,.
\end{equation}
For a proof of these results, and  for a different perspective on completions, as a localisation, see \cite{Dwyer/Greenlees:2002a}, and also \cite[Tag091N]{StacksProject}.

The result below is a crucial step in the proof of Theorem~\ref{th:ca-ring}. 

\begin{proposition}
\label{pr:ca-ring}
Let $A$ be a local ring with maximal ideal $\fm$ and residue field $k$, and  let $\wh A$ be the $\fm$-adic completion of $A$. The following statements hold for any object $X\in \dcat A$ that is $\fm$-complete.
\begin{enumerate}[\quad\rm(1)]
\item
If $H(k\lotimes_AX)$ is bounded, then  the natural map 
\[
X \lra k\lotimes_{A}X 
\]
induced by the surjection $A\to k$, is nonzero in homology.
\item
If $\rank_kH(k\lotimes_AX)$ is finite, then $X$ is in $\Thick(\wh A)$. 
\item
If $H(k\lotimes_AX)\cong \Si^sk$ for some integer $s$,  then $X\cong \Si^s \wh A$. 
\end{enumerate}
\end{proposition}

\begin{proof}
(1) Since $X$ is  $\fm$-complete so is $H^n(X)$ for each $n$. Thus, if $\fm\cdot H^n(X) = H^n(X)$, then $H^n(X)=0$; see \cite[1.4]{Simon:1990a}, and also \cite[Tag09b9]{StacksProject}. Given this observation, the hypothesis that $H(k\lotimes_AX)$ is bounded implies $H(X)$ is bounded; this can be checked via a standard devissage argument using $H(K\otimes_RM)$, where $K$ is the Koszul complex of $R$.  Set $i=\inf\{n\mid H^n(X)\ne 0\}$. Then the composed map
\[
H^i(X) \lra H^i(k\lotimes_AX) \cong k \otimes_A H^i(X) \cong H^i(X)/\fm H^i(X)\,,
\]
where the first isomorphism holds because the tensor product is right exact, is the obvious surjection and the target is nonzero. This justifies the claim.

(2) We verify this by an induction on the integer $r\colonequals \rank_k H(k\lotimes_AX)$.  The base case is $r=0$. Then $k\lotimes_AX= 0$ in $\dcat A$,  that is to say, $\fm$ is not in $\supp_AX$. Thus $\rgam_{\fm}X\cong 0$. It remains to note that
\[
X\cong \llam^{\fm} X \cong \llam^{\fm} \rgam_{\fm} X \cong 0\,,
\]
where the second isomorphism is from \eqref{eq:gam-lam}.

Suppose $r\ge 1$.  Since $ \Hom_{\dcat A}(\Si^{i}A, -)\cong H^{-i}(-)$,  part (1) is equivalent to the existence of  a map $\Si^s A\to X$ in $\dcat A$ such that the induced map $\Si^s k\to H(k\lotimes_AX)$ is nonzero. Since $X$ is $\fm$-complete, the map $\Si^s A\to X$ factors through $\Si^s \wh A\to X$ and this fits into an exact triangle
\[
\Si^s \wh A\lra X\lra Y \lra \Si^{s+1} \wh A\lra\,.
\]
Evidently $\rank_k H(k\lotimes_AY)=r-1$, so the induction hypothesis yields that $Y$ is in $\Thick(\wh A)$, and hence so is $X$.

(3) When $r=1$, the argument above yields that $\Si^s \wh A\cong X$, as desired.
\end{proof}

\subsection*{A derived Morita equivalence}
Let $A$ be a local ring with maximal ideal $\fm$.  It helps to consider another  adjoint pair:  The  map $A \to \RHom_A(\rgam_\fm A,\rgam_\fm A)$ induces, by \eqref{eq:gam-lam},  a quasi-isomorphism 
\begin{equation}
\label{eq:completion}
\wh A\longrightarrow \RHom_A(\rgam_\fm A,\rgam_\fm A)
\end{equation}
so derived Morita theory  yields adjoint functors
\[
\begin{tikzcd}[column sep = large]
\dcat{\wh A} \arrow[leftarrow,yshift=-1ex,swap,rr, "{\RHom_A(\rgam_\fm A,-)}"]
 	&& \arrow[leftarrow,yshift= 1ex,ll,swap, "{\rgam_{\fm}A\lotimes_{\wh A}-}"] \dcat A \,.
\end{tikzcd}
\]
The functors introduced above give alternative descriptions of the category we are interested in, namely, the thick subcategory generated by $\rgam_{\fm}A$. 

\begin{lemma}
\label{le:GM-equivalence}
The adjoint pairs above restrict to triangle equivalences 
\[
\begin{tikzcd}[column sep=large]
\Thick_{\wh A}(\wh A) \arrow[rr, rightarrow, yshift=.75ex, "{\rgam_{\fm}A\lotimes_{\wh A}-}"] 
	&& \arrow[ll, rightarrow,yshift=-.75ex, "{\RHom_A(\rgam_{\fm}A,-)}"]
		\Thick_A(\rgam_\fm A) \arrow[rightarrow,yshift=-.75ex,swap,rr,"\llam^\fm"]
  	&&  \arrow[rightarrow,yshift=.75ex,ll,swap,"\rgam_\fm"] \Thick_A(\wh A)\,.
\end{tikzcd}
\]
Moreover the pair on the left is compatible with tensor products.
\end{lemma}

\begin{proof}
The equivalence on the left follows by the usual argument in Morita theory, given \eqref{eq:completion}.  The equivalence on the right is by the isomorphisms \eqref{eq:gam-lam} for $I=\fm$.
\end{proof}

Composing the equivalences in Lemma~\ref{le:GM-equivalence} yields a triangle equivalence 
\[
\Thick_{\wh A}(\wh A) \longiso \Thick_A(\wh A)\,.
\]
It is easily verified that this is induced by the restriction functor $\bfD(\wh A)\to \dcat A$ arising from the natural map $A\to \wh A$ of rings. 

\begin{proof}[Proof of Theorem~\ref{th:ca-ring}]
We may assume $(A,\fm,k)$ is a local ring and $\fp=\fm$, so that $k(\fp)=k$. Thus $X$ is an $\fm$-torsion $A$-complex.  We recall that $\dcat A$ is a tensor triangulated category, generated by its unit $A$, and so compact objects and dualisable objects in $\dcat A$ coincide. This fact will be used throughout the proof. 

(1)$\Rightarrow$(2): Let $K$ be the Koszul complex on a generating set for the ideal $\fm$.  As $X$ is dualisable the $A$-complex $K\otimes_A X$ is compact, by Proposition~\ref{pr:dualisable}, and so in $\Thick(A)$. Hence the $k$-vector space $H(k\lotimes_A (K\otimes_AX))$ has finite rank. Since $k$ is a field there are isomorphisms
\begin{align*}
H(k\lotimes_A (K\otimes_AX)) 
	&\cong H( (k\otimes_A K)\otimes_k (k\lotimes_AX)) \\
	&\cong H(k\otimes_A K)\otimes_k H(k\lotimes_AX)\,.
\end{align*}		
Observe that $H(k\otimes_AK)$ is nonzero. As the rank of $H(k\lotimes_A (K\otimes_AX))$ is finite, so is that of $H(k\lotimes_AX)$.

(2)$\Rightarrow$(3):  Since $k$ is $\fm$-torsion, the natural map below is an isomorphism:
\[
k\lotimes_A X \longiso k \lotimes_A \llam^{\fm} X 
\]
The hypothesis and Proposition~\ref{pr:ca-ring}  imply that $\llam^{\fm} X$ is in $\Thick(\wh A)$.  Lemma~\ref{le:GM-equivalence} then yields that $\rgam_{\fm} X$ is in $\Thick(\rgam_{\fm }A)$. It remains to recall that $X$ is $\fm$-torsion. 

(3)$\Rightarrow$(1): As $\rgam_{\fm}A$ is the unit of $\rgam_{\fm}\dcat A$, it is dualisable. It remains to note that the dualisable objects form a thick subcategory.

The last part of the theorem follows from Proposition~\ref{pr:ca-ring}(3).
\end{proof}

\subsection*{Balmer spectrum}
Set $\cat T\colonequals \dcat A$ and fix a prime $\fp$ in $\Spec A$. The full subcategory $\duals{(\Gamma_{\fp}\cat T)}$ of dualisable objects in  $\Gamma_{\fp}\cat T$ is an essentially small tensor triangulated category,  with unit $\rgam_{\fp}(A_\fp)$.  The unit generates $\Gamma_{\fp}\cat T$, in the sense of thick subcategories, so thick subcategories are tensor ideal; this follows from Theorem~\ref{th:ca-ring}.

We are interested in the  lattice of thick subcategories of $\duals{(\Gamma_{\fp}\cat T)}$, captured in the Balmer spectrum introduced in ~\cite{Balmer:2005a}. Given Theorem~\ref{th:ca-ring} and Lemma~\ref{le:GM-equivalence}, one can describe the underlying topological space easily. 

\begin{corollary}
\label{co:balmer}
One has a homeomorphism $\bSpec \duals{(\Gamma_{\fp}\cat T)} \cong \Spec(\wh{A}_\fp)$. 
\end{corollary}

\begin{proof}
We can again assume $A$ is local with maximal ideal  $\fp$.   Given Theorem~\ref{th:ca-ring}, the equivalence of categories on the left in Lemma~\ref{le:GM-equivalence} yields an homeomorphism
\[
\bSpec \duals{(\Gamma_{\fp}\cat T)} \cong \bSpec \Thick_{\wh A}(\wh A)\,.
\]
It remains to recall the classification of the thick subcategories of  perfect complex of a commutative noetherian ring, due to Hopkins~ \cite{Hopkins:1987a} and Neeman~\cite{Neeman:1992a}, interpreted in terms of the Balmer spectrum~\cite[Theorem~5.5]{Balmer:2005a}.
\end{proof} 

\begin{remark}
\label{re:completion}
The (Zariksi) spectrum of $\wh{A}_\fp$ can be wildly different from that of $A_\fp$, though they have the same Krull dimension. We offer a few remarks to convey this point. Suppose $A$ is local and $\fp=\fm$, the maximal ideal of $A$. The completion map $A\to \wh A$ induces a homomorphism
\[
\Spec \wh A\lra \Spec A\,.
\]
This map is surjective as $A\to \wh A$ is faithfully flat. Moreover $\dim A=\dim \wh A$. Since $\fm \wh A$ is the maximal ideal of $\wh A$, there is a single point lying over the closed point $\fm$ of $\Spec A$, namely, the closed point of $\Spec \wh A$.  This shows that the Krull dimension of the fibres of the completion map is at most $\dim A - 1$. 

The fibres over other non-closed points can be highly non-trivial. This is so  even over the generic points of $\Spec A$.  It is easy to construct local domains such that the generic formal fibre has more than one point. Here is one example: Consider the local ring
\[
A\colonequals \frac{\mathbb{Q}[x,y]_{(x,y)}}{(x^2 - y^2(y-1))}\,.
\]
Since $x^2-y^2(y-1)$ is irreducible in the ring $\mathbb{Q}[x,y]_{(x,y)}$, the ring $A$ is a domain. However that polynomial factors in the $(x,y)$-adic completion $\mathbb{Q}[\!|x,y|\!]$, so the completion of $A$ is not a domain.

Here is more drastic scenario: Given any pair of integers $d,t$ with $0< t < d-2$, Rotthaus~\cite{Rotthaus:1991a} constructs a noetherian local domain $A$ of Krull dimension $d$ such that the formal fibre over the generic point of $A$ has Krull dimension $t$.
\end{remark}

\subsection*{Reflexive objects} With $A$ and $\fp$ as before, the Spanier--Whitehead dual of an object $X$ in $\Gamma_{\fp}\dcat A$ is 
\[
\swd(X) =  \rgam_{\fp}\RHom_A(X,\rgam_{\fp}(A_\fp))\,.
\]
Recall, from Section~\ref{se:dualisability}, that $X$ is \emph{reflexive} if the natural map $X\to \swd\swd(X)$ is an isomorphism.
Here is the connection between this notion and dualisability.

\begin{lemma}
Let $A$ be a commutative noetherian ring, $\fp$ a prime ideal in $A$, and fix $X$ in $\Gamma_{\fp}\dcat A$. If $X$ is dualisable, it is reflexive; the converse holds when the local ring $A_\fp$ is regular and $H(X)$ is bounded.
\end{lemma}

\begin{proof}
The first part of the statement follows from Proposition~\ref{pr:dualisable}. So it remains to prove that when $A_\fp$ is regular,  $H(X)$ is bounded, and $X$ is reflexive, it is dualisable.

We can replace $A$ by $A_\fp$ and assume it is a regular local ring, say with maximal ideal $\fm$ and residue field $k$. The Spanier--Whithead duality on the category $\Gamma_{\fm}\dcat A$ is the functor
\[
\swd (X) \colonequals \rgam_\fm \RHom_A(X,\rgam_\fm A) \cong \rgam_\fm \RHom_A(X,A)\,.
\]
Thus, keeping in mind Greenlees--May duality~\eqref{eq:GM} one gets that
\begin{align*}
\swd\circ \swd (X) 
	& \cong \rgam_\fm\RHom_A(\rgam_\fm \RHom_A(X,A),A) \\
	& \cong \rgam_\fm\RHom_A(\RHom_A(X,A),\llam^\fm A) \\
	&\cong \rgam_\fm\llam^\fm \RHom_A(\RHom_A(X,A),A)\\
	&\cong \rgam_\fm \RHom_A(\RHom_A(X,A),A)
\end{align*}
Since $A$ is regular, the $A$-module $k$ has a finite free resolution, so the isomorphism $X\iso \swd\circ \swd(X)$, which holds because $X$ is reflexive,  induces isomorphisms
\begin{align*}
k\lotimes_AX 
	&\longiso k\lotimes_A \rgam_\fm \RHom_A(\RHom_A(X,A), A) \\
	&\longiso \RHom_k(\RHom_k(k\lotimes_AX,k),k)
\end{align*}
In homology, this yields that the natural vector-space duality is an isomorphism:
\[
H(k\lotimes_AX)\cong \Hom_k(\Hom_k(H(k\lotimes_AX),k))\,.
\]
Hence $\rank_k H^i(k\lotimes_AX)$ is finite for each $i$. As $A$ is regular, $k$ is in $\Thick(A)$, thus $H(X)$ bounded implies $H(k\lotimes_AX)$ is bounded as well. We deduce that $\rank_k H(k\lotimes_AX)$ is finite, so $X$ is dualisable, by Theorem~\ref{th:ca-ring}. 
\end{proof}

In the preceding result, the condition that $H(X)$ is bounded is required: When $A$ is any local ring,  $X\colonequals \bigoplus_i \Sigma^iA$ is reflexive but not dualisable, for it is not compact.

\subsection*{Strong generation}
Let us return to the general framework of a compactly generated tensor triangulated category $({\cat T},\otimes,\one)$. We are interested in the property that $\cat T^{\mathrm c}$, the thick subcategory consisting of compact objects, has a \emph{strong generator}, in the sense of Bondal and Van den Bergh~\cite{Bondal/Vandenbergh:2003a}. Roughly speaking, an object $G\in \cat T^{\mathrm c}$ is a strong generator if there exists an integer $d$ such that every compact object in $\cat T$ can be built out of $G$ using direct sums, retracts, and at most $d$ extensions. This might be viewed as a regularity condition, for when $A$ is a commutative noetherian ring the category of perfect $A$-complexes  ${\dcat A}^{\mathrm c}$ has a strong generator if and only if the global dimension of $A$ is finite; see \cite[Proposition~7.2.5]{Rouquier:2008a}. 

A question that arises is this: If $\cat T^{\mathrm c}$ has a strong generator, does each category of local dualisable objects also have a strong generator?  The motivation comes from the following result in commutative algebra; we recall that $A_\fp$ is regular precisely when the subcategory of compact objects in $\dcat{A_\fp}$ has a strong generator.

\begin{corollary}
\label{co:regularity}
Let $A$ be a commutative noetherian ring and $\fp$ a prime ideal in $A$. When $A_\fp$ is regular, $\rgam_{\fp}(A_\fp)$ is a strong generator for  the subcategory of dualisable objects among the $\fp$-local $\fp$-torsion $A$-complexes.
\end{corollary}

\begin{proof}
  We pass to the localisation at $\fp$ and assume $A$ is a regular
  local ring, and hence of finite global dimension.  Then $\wh A$, the
  completion of $A$ at its maximal ideal also has finite global
  dimension; see \cite[Proposition~2.2.2]{Bruns/Herzog:1998a}. Thus
  $\wh A$ is a strong generator for $\Thick_{\wh A}(\wh A)$. It
  remains to recall that this category is triangle equivalent to the category of dualisable objects
  in $\Gamma_{\fp}\dcat A$, by Theorem~\ref{th:ca-ring} and
  Lemma~\ref{le:GM-equivalence}.
\end{proof}

\section{Other contexts}
\label{se:others}
In this section we discuss  other examples of compactly generated tensor triangulated categories for which we have some information on the local dualisable objects.

\subsection*{Noetherian schemes}
\label{ss:schemes}
Let $\bbX$ be a separated noetherian scheme and $\cat T$  the derived category of quasi-coherent sheaves on $\bbX$, viewed as a tensor triangulated category in the usual way.  For each $x\in \bbX$ one can consider the dualisable objects in the subcategory $\Gamma_{x}\cat T\subseteq \cat T$ consisting of objects supported on $\{x\}$. This category is described by Theorem~\ref{th:ca-ring}, for by standard arguments it is the same as the dualisable objects in $\Gamma_{\fm}\dcat{\mathcal{O}_{\bbX,x}}$, where $\mathcal{O}_{\bbX,x}$ is the  local ring at $x$ and $\fm$ is its maximal ideal. Thus  Corollary~\ref{co:balmer} and Remark~\ref{re:completion} yield the following result.

\begin{corollary}
The Balmer spectrum of $\duals{(\Gamma_{x}\cat T)}$ is homeomorphic to $\Spec \wh{\mathcal{O}}_{\bbX,x}$. \qed
\end{corollary}

\subsection*{Modular representations of finite groups}
Let $k$ be a field of positive characteristic and $G$ a finite group
whose order is divisible by the characteristic of $k$. We write
$\StMod kG$ for the stable category of $kG$-modules, and $\stmod kG$
for its full subcategory  of finite dimensional modules. Then $\StMod
kG$ is a compactly generated, with compact objects $\stmod kG$, and
tensor product over $k$, with diagonal $G$-action, gives it a
structure of a tensor triangulated category. The  unit is $k$ with
trivial action and the  function object is $\Hom_k(-,-)$, again with
the diagonal $G$-action. Moreover compact objects in $\StMod kG$ are easily seen to be dualisable and hence one has an equality $\comp{(\StMod kG)}=\duals{(\StMod kG)}$.  

The group cohomology ring $H^*(G,k)$ is a finitely generated $k$-algebra. As in the case of the derived category of a commutative noetherian ring, one  considers the subcategory $\Gamma_{\fp}(\StMod kG)$ of the (big) stable module category consisting of $\fp$-local and $\fp$-torsion modules. These are the minimal localising tensor ideals of $\StMod G$, and so the lattice of localising tensor ideals in the stable module category  are parameterised by subsets of $\Proj H^*(G,k)$,  the homogenous prime ideals in $H^*(G,k)$  not containing the maximal ideal $H^{\geqslant 1}(G,k)$. These results are proved in \cite{Benson/Iyengar/Krause:2011b}; see also \cite{Benson/Iyengar/Krause/Pevtsova:2017a}.  In \cite{Benson/Iyengar/Krause/Pevtsova:2023b} we prove the following analogue of Theorem~\ref{th:ca-ring}; the case when $\fp$ is a closed point is also treated  in the work of Carlson~\cite{Carlson:2210.01842v2}.

\begin{theorem}
\label{th:stmod}
Fix $\fp$ in $\Proj H^*(G,k)$.  For each $kG$-module $X$ in $\Gamma_\fp(\StMod kG)$ the following conditions are equivalent:
\begin{enumerate}[\quad\rm (1)]
\item
$X$ is dualisable in  $\Gamma_\fp(\StMod kG)$;
\item
The $H^*(G,k)_\fp$-module $H^*(G,C\otimes_kX)_\fp$ is artinian for each finite dimensional $kG$-module $C$;
\item 
$M$ is in $\Thick(\Gamma_\fp(\StMod kG))$.\qed
\end{enumerate}
\end{theorem}

Compare condition (2) above with  the corresponding condition in Theorem~\ref{th:ca-ring}. It is not hard to prove that the latter implies that the $A_\fp$ module $H(C\lotimes_AX)$ is artinian for each compact object (that is to say perfect complex)  in $\dcat A$; see \cite{Benson/Iyengar/Krause/Pevtsova:2023b}. But condition Theorem~\ref{th:ca-ring}(2) is strictly stronger, for the residue field is not a compact object in $\Gamma_{\fp}\dcat A$ unless the local ring $A_\fp$ has finite global dimension.  This suggests that there is a broader framework than that covered by Theorem~\ref{th:ca-ring} wherein one can get a handle on dualisable objects.

\subsection*{The stable homotopy category}
The last example we consider is the stable homotopy category of spectra. This is a rather more involved context than the ones discussed earlier, so the discussion is more telegraphic than before; we refer readers to \cite{Hovey/Strickland:1999a} for details.

Akin to the derived category of a commutative ring, the stable homotopy category is determined  by its localisations at various prime numbers. Fix a prime number $p$, a positive integer $n$, and let $\cat S$ be the homotopy category of $p$-local spectra. This is a tensor triangulated category with tensor identity the $p$-local sphere $S$. Let $K(n)$ be the Morava $K$-theory of level $n$ at the prime $p$, and $\cat K$ the category of $K(n)$-local spectra. By \cite[Theorem~7.5]{Hovey/Strickland:1999a}, this is a minimal localising subcategory of  $\cat S$. Let $\hat L \colon \cat S \to \cat K$ be the localisation functor.

\begin{theorem}
\label{th:stmod}
Fix $X$ in $\cat K$,  and consider the following conditions:
\begin{enumerate}[\quad\rm (1)]
\item
$X$ is dualisable in  $\cat K$;
\item
$K(n)_*(X)$ is finite;
\item 
$X$ is in $\Thick(\wh L S)$.
\end{enumerate}
Then \emph{(1)} and \emph{(2)} are equivalent and are implied by \emph{(3)}. \qed 
\end{theorem}

Condition (3) is strictly stronger than (1) and (2): Hopkins
constructed a $K(n)$-local spectrum $Y$ in the case $n=1$ that is
dualisable but not finitely built from $\wh L S$.  Set $E\colonequals
\wh{E(1)}$ and $T\colonequals \psi^a-1\in E^0(E)$, where $\psi^a$ is
the Adams psi-operation with $a$ a topological generator for
$1+p\bbZ_p$. Provided $p$ is odd, $Y$ is the cofibre of the map
$T^2-p\colon E \to E$.  The spectrum $Y$ is dualisable but is not in
the thick subcategory generated by the Picard group of invertible
objects in $\cat K$, and hence not in $\Thick(\wh L S)$.
For details, see  \cite[Section~15.1]{Hovey/Strickland:1999a}.

\bibliographystyle{amsplain}

\newcommand{\noopsort}[1]{}
\providecommand{\bysame}{\leavevmode\hbox to3em{\hrulefill}\thinspace}
\providecommand{\MR}{\relax\ifhmode\unskip\space\fi MR }
\providecommand{\MRhref}[2]{%
  \href{http://www.ams.org/mathscinet-getitem?mr=#1}{#2}
}
\providecommand{\href}[2]{#2}

\end{document}